\documentclass[12pt,reqno]{amsart}
\usepackage{amsmath}
\usepackage{amsfonts}
\usepackage{amssymb,latexsym}
\usepackage{cite}
\usepackage[height=190mm,width=160mm]{geometry}
\usepackage{hyperref}
\theoremstyle{plain}
\newtheorem{theorem}{Theorem}
\newtheorem{lemma}{Lemma}
\newtheorem{corollary}{Corollary}
\newtheorem{proposition}{Proposition}

\theoremstyle{definition}
\newtheorem{definition}{Definition}
\theoremstyle{remark}

\numberwithin{equation}{section}
\usepackage{xcolor}

\begin{document}
\title[On bi-periodic Padovan and Perrin
quaternions]{On bi-periodic Padovan and Perrin
quaternions over finite fields}
\author{Diana SAVIN}
\address{Department of Mathematics and Computer Science, Transilvania
University of Brasov, 500091, Romania}
\email{diana.savin@unitbv.ro}
\author{Elif TAN}
\address{Department of Mathematics, Faculty of Science, Ankara University
06100 Tandogan Ankara, Turkey}
\email{etan@ankara.edu.tr}

\begin{abstract}
In this paper, we investigate bi-periodic Padovan and bi-periodic Perrin quaternions in the quaternion algebra \(Q_{\mathbb{Z}_p}\). We introduce the bi-periodic Perrin sequence and clarify its structural relationship with the bi-periodic Padovan sequence. By extending these sequences to the quaternion setting, we analyze their norm properties in the modular framework. For suitable choices of twin prime coefficients, we derive explicit criteria characterizing zero divisors and invertible elements in \(Q_{\mathbb{Z}_p}\).

\end{abstract}

\maketitle

\section{Introduction}

The Fibonacci sequence, denoted by $\{F_n\}_{n\ge 0}$, is defined by the second-order recurrence relation $F_{n}=F_{n-1}+F_{n-2},$ $n\geq 2$, with the initial conditions $F_{0}=0$ and $F_{1}=1$. 
%Its companion sequence, Lucas sequence $\left \{ L_{n}\right \} $ satisfies the same recurrence relation, but begins with the initial conditions $L_{0}=2$ and $L_{1}=1$.
This simple yet powerful recursive definition provides a fundamental model for describing various natural phenomena associated with growth and development. %Owing to its mathematical elegance and applicability, the Fibonacci sequence has found profound connections across a wide range of disciplines, including mathematics, computer science, physics, chemistry, biology, and even sociology. 
The study of generalizations of the Fibonacci sequence and higher-order recurrence relations has attracted considerable attention in the literature since they found profound connections across a wide range of disciplines, including mathematics, computer science, physics, chemistry, biology, and even sociology. One of the most notable third-order integer sequences in this context is
the Padovan sequence \cite{padovan}. The Padovan sequence $\{p_n\}_{n\ge 0}$, listed as A000931 in the On-Line Encyclopedia of Integer Sequences \cite{oeis}, is defined by $p_{n}=p_{n-2}+p_{n-3},$ $n\geq 3$, with the initial terms $p_{0}=1,p_{1}=0,$
and $p_{2}=1$. The first few terms of the sequence are:
$$1,\ 0,\ 1,\ 1,\ 1,\ 2,\ 2,\ 3,\ 4,\ 5,\ 7,\ 9,\ 12,\ 16,\ 21,\ 28,\ 37,\ 49,\ 65,\ 86,\ 114, \ldots .$$
The Perrin sequence $\{r_n\}_{n\ge 0}$, listed as sequence A001608 in \cite{oeis}, satisfies the same recurrence
relation as Padovan sequence but begins with the initial values
$r_0=3,\ r_1=0, \ \text{and} \ r_2=2.$
The first few terms of the sequence are:
$$
3,\,0,\,2,\,3,\,2,\,5,\,5,\,7,\,10,\,12,\,17,\,22,\,29,\,39,\,51,\,68,\,90,\,119,\,158,\,209,\,277,
 \ldots .$$
Padovan and Perrin sequences are related by the well-known relation $r_n = 3p_{n-3}+2p_{n-2}, \ n\ge 3.$
These sequences also have rich connections to different areas of mathematics, including graph theory and discrete structures. Their interplay with graph-based models has been investigated in several contexts; for a recent contribution in this direction, see \cite{padovangraph}. A detailed exposition of the Fibonacci sequence and related integer sequences, we refer to the books \cite{koshy1, koshy2}.

On the other hand, quaternions, originally introduced by Hamilton, can be
viewed as an extension of complex numbers and have found applications in
various fields such as computer science, physics, differential geometry, and
quantum mechanics. Let $F$ be a field with characteristic not $2.$ The generalized quaternion algebra over a field $F$ is defined as:
\begin{equation}
Q_{F}\left( s,t\right) =\left \{ x+yi+zj+wk\mid x,y,z,w\in
F,i^{2}=s,j^{2}=t,ij=-ji=k\right \}   \label{q}
\end{equation}
where $s$ and $t$ are nonzero elements of $F$. It is known that
the algebra $Q_{\mathbb{R}}\left( -1,-1\right)$ is the classical real quaternion algebra. Let $B$ be a finite-dimensional algebra. Then $B$ is a division algebra if and only if it has no zero divisors; that is, for any nonzero $x,y\in B$, we have $xy\neq 0$. If $H$ is a generalized quaternion algebra, then $H$ is a division algebra if and only if a quaternion with zero norm is necessarily the zero quaternion. In other
words, for $X\in Q_{F}\left( s,t\right) ,$ the norm of $X,$ denoted as $N(X)$
and defined as $N(X)=x^{2}-sy^{2}-tz^{2}+stw^{2}$ equals zero if and only if 
$X=0$. Otherwise, the algebra is called a split algebra. In particular, while the real quaternion algebra is a division algebra, the quaternion
algebra over the finite field $\mathbb{Z}_{p},$ denoted by $Q_{\mathbb{Z}_{p}}\left( -1,-1\right)$ and abbreviated as $Q_{\mathbb{Z}_p}$, is known to be a split algebra for any odd prime $p$; see \cite{grau}. Special elements in quaternion algebras over finite fields have also been studied in \cite{savin, savin0, savin1, mi, savinalgebra, savinoct}, to which we refer for further background.

In recent years, there has been growing interest in quaternion sequences defined via special integer sequences such as the Fibonacci, Lucas, Pell, Padovan, and Perrin sequences \cite{diskayanenken1, diskayanenken2}. The algebraic properties of these quaternion sequences over various quaternion algebras have been investigated by many authors. In particular, Horadam~\cite{horadam} defined the Fibonacci quaternions over the real quaternion algebra as
\[
Q_n = F_n + F_{n+1}i + F_{n+2}j+ F_{n+3}k,
\]
where \(F_n\) denotes the \(n\)-th Fibonacci number. Padovan quaternions were later introduced by Ta\c{s}c{\i}~\cite{tasci}. Further studies on quaternions associated with special integer sequences can be found in \cite{chaos1, chaos2, chaos3, isbilir, halici0, halici1, lee, tan-rosa, flautsavin, horadam, savin2019}. For quaternion sequences defined over finite fields, we refer the reader to \cite{savin-tan, tan-savin}.

In this paper, we consider the bi-periodic Padovan sequence $\{P_n\}_{n\ge 0}$, studied by Diskaya and Menken \cite{diskaya}, which satisfies the following recurrence relation:
\[P_n =\begin{cases}a P_{n-2} + P_{n-3}, & \text{if } n \text{ is even}, \\b P_{n-2} + P_{n-3}, & \text{if } n \text{ is odd},\end{cases}\quad \text{for } n \geq 3,\]with the initial values $P_0 = 1,\ P_1 = 0,\ \text{and} \ P_2 = a$.
It should be noted that Alp et al.~\cite{ternary} studied a general two-periodic ternary recurrence of the form
\[
\gamma_n =
\begin{cases}
a_1\gamma_{n-1}+a_2\gamma_{n-2}+a_3\gamma_{n-3},
& \text{if } n \text{ is even},\\
b_1\gamma_{n-1}+b_2\gamma_{n-2}+b_3\gamma_{n-3},
& \text{if } n \text{ is odd},
\end{cases}
\]
and derived the following sixth-order recurrence relation
\[
\gamma_n
=(a_1b_1+a_2+b_2)\gamma_{n-2}
+(a_1b_3-a_2b_2+a_3b_1)\gamma_{n-4}
+a_3b_3\gamma_{n-6}.
\]
Setting
$a_1=b_1=0,\ a_2=a,\ b_2=b,\ a_3=b_3=1,$
one obtains the recurrence relation
\begin{equation}
P_n=(a+b)P_{n-2}-abP_{n-4}+P_{n-6},\quad n\geq 6,
\end{equation} \label{rec1}which coincides with the recurrence underlying the bi-periodic sequence considered in the present paper. Our interest, however, lies in the specific Padovan-type structure arising from this choice and the corresponding quaternion sequences over finite fields.

From \cite{diskaya}, the generating function of the bi-periodic Padovan sequence is given by
\begin{equation}
G(x)=\frac{1-bx^{2}+x^{3}}{1-(a+b)x^{2}+abx^{4}-x^{6}}. \label{gf}
\end{equation}
Recently, G{\"u}ng{\"o}r et al. \cite{gungor} provided a combinatorial interpretation of these numbers using a weighted tiling approach.

While bi-periodic Padovan numbers have already been studied in the literature, an analogous bi-periodic extension of the Perrin sequence has not yet been systematically investigated. One of the main objectives of this paper is to introduce the \emph{bi-periodic Perrin sequence} and to establish its fundamental relationship with the bi-periodic Padovan sequence. Furthermore, another key motivation of this study is to consider the quaternion sequences associated with these recurrences, namely the bi-periodic Padovan and bi-periodic Perrin quaternions, and to examine the existence of zero divisors in the quaternion algebra \(Q_{\mathbb{Z}_p}\).

Understanding zero divisors in such split quaternion algebras over finite fields is a problem of intrinsic algebraic interest, linking number theoretic properties of sequences with the algebraic structure of quaternions. As a first step, we focus on the special modular setting where the twin prime coefficients $a=p-2$ and $b=p$ for a prime $p\ge 5$. It is worth emphasizing that the approach adopted in the Fibonacci quaternion case studied by Savin \cite{savin}, as well as in the Leonardo quaternion setting considered in \cite{savin-tan, tan-savin}, cannot be directly transferred to the bi-periodic Padovan framework. In those cases, the underlying recurrences lead to norm expressions with a relatively uniform structure, allowing zero divisors and invertible elements to be characterized without imposing strong restrictions on the defining parameters. By contrast, the bi-periodic nature of the Padovan and Perrin sequences introduces alternating coefficients and non-consecutive index interactions, which significantly complicate the norm structure. This structural difference motivates the use of specific coefficient choices, such as twin prime pairs. In particular, the choice of the twin-prime coefficients $a=p-2$ and $b=p$ results in a simpler modular structure for the bi-periodic recurrence and its associated quaternion norms. This choice makes the zero-divisor conditions sufficiently tractable to obtain explicit characterizations for a general prime $p$, while retaining the essential bi-periodic structure of the sequences. On the other hand, in those papers the zero divisors were determined only for certain special values of the prime $p$, whereas the approach adopted here applies to a general prime $p$, which we believe constitutes one of the main strengths of this paper. In this context, we explicitly determine which bi-periodic Padovan and Perrin quaternions are zero divisors in 
$Q_{\mathbb{Z}_p}$. We then provide a detailed analysis of the zero divisor structure in the specific cases $p=7$, $p=13$, and $p=181$ for the bi-periodic Perrin quaternions. In particular, we show that for $p=13$, all bi-periodic Perrin quaternions are invertible.

%%%%%%%%%%%%%%%%%%%%%%%%%%%%%%%%%%%%%%%%%%%5
\section{Bi-periodic Padovan quaternions over finite fields}

In this section, we introduce bi-periodic Padovan quaternions and present some of their basic properties, including the recurrence relation and generating function. Then, by fixing a specific twin prime parameterization, we study their norm structure and derive criteria for the existence of zero divisors in \(Q_{\mathbb{Z}_p}\).

\begin{definition}
The bi-periodic Padovan quaternion sequence $\{QP_n\}_{n\ge 0}$ is defined as
\begin{equation*}
QP_{n}=P_{n}+P_{n+1}i+P_{n+2}j+P_{n+3}k
\end{equation*}
where $P_{n}$ is the $n$-th bi-periodic Padovan number.
\end{definition}
The first few terms can be seen as follows:
\begin{eqnarray*}
QP_{0} &=&1+aj+k, \\
QP_{1} &=&ai+j+a^{2}k, \\
QP_{2} &=&a+i+a^{2}j+(a+b)k, \\
QP_{3} &=&1+a^{2}i+(a+b)j+(a^{3}+1)k, \\
QP_{4}
&=&a^{2}+(a+b)i+(a^{3}+1)j+(a^{2}+ab+b^{2})k,
\\
QP_{5}
&=&a+b+(a^{3}+1)i+(a^{2}+ab+b^{2})j+((a+b)+a(a^{3}+1))k.
\end{eqnarray*}

Now we derive the quaternion analogues of the identities satisfied by the bi-periodic Padovan numbers, as given in relations \eqref{rec1} and \eqref{gf}. The bi-periodic Padovan quaternions satisfy the following recurrence relation:
\begin{equation} \label{rec0}
Q{P_{n}}=(a+b)QP_{n-2}-ab\,QP_{n-4}+QP_{n-6}{,\  \ n\geq }{6.}
\end{equation}

From the recurrence relation (\ref{rec0}), the generating function of the bi-periodic Padovan quaternion sequence is given by
\[
G(x)=\sum_{n=0}^{\infty}QP_{n}x^{n}
=
\frac{
A(x)+B(x)i+C(x)j+D(x)k
}{
1-(a+b)x^{2}+abx^{4}-x^{6}
}
\]
where
\[
\begin{aligned}
A(x) &= 1-bx^{2}+x^{3},\\
B(x) &= ax+x^{2}-abx^{3}+x^{5},\\
C(x) &= a+x-abx^{2}+x^{4},\\
D(x) &= 1+a^{2}x+(1-a^{2}b)x^{3}+ax^{5}.
\end{aligned}
\]

%%%%%%%%%%%%%%%%%%%%%%%%%%%%%%%%%%%%%%%%%%%%%%%%
Next, we consider the quaternion algebra \(Q_{\mathbb{Z}_p} = Q_{\mathbb{Z}_p}(-1,-1)\) over the finite field \(\mathbb{Z}_p\). We focus on bi-periodic Padovan quaternions with the twin prime coefficient choice \(a=p-2\) and \(b=p\), and analyze their norm structure in the modular setting. Using these norms, we determine the zero divisors in \(Q_{\mathbb{Z}_p}\).

From the definition of bi-periodic Padovan sequence with coefficients \(a=p-2\) and \(b=p\) for odd indices, we have
$P_{2k+1}=pP_{2k-1}+P_{2k-2}\equiv P_{2k-2}\pmod{p}.$ Replacing \(k\) by \(k+1\) yields
\begin{equation}
P_{2k}\equiv P_{2k+3} \pmod {p} \label{p}
\end{equation}

We begin with a useful result that will help us determine the zero divisors. To do this, first we give the following lemma.
\begin{lemma} \label{l}For the bi-periodic Padovan sequence with twin prime coefficients $a=p-2$ and $
b=p$, we have
 \begin{equation*}
    P_{2k} \equiv (-1)^k \sum_{i=0}^{\lfloor k/3 \rfloor} (-1)^i \binom{k-2i}{i} 2^{\,k-3i} \pmod{p}.
\end{equation*}   
\end{lemma}
\begin{proof}
 The proof can be done by mathematical induction over $k$. For initial terms we have $P_{0} \equiv 1 \pmod{p}$, $P_{2} \equiv p-2 \pmod{p}$, and $P_{4} \equiv (p-2)^2 \pmod{p}$, which agree with the stated formula. Assume that for some \(k\ge 2\), the statement holds for every integer \(j\le k\). Now we prove that the formula holds for $k+1$. Using the recurrence relation of bi-periodic Padovan sequences for even indices, we have
$P_{2k+2}\equiv -2P_{2k}+P_{2k-1}\pmod{p}$. From (\ref{p}), we have $P_{2k-1}\equiv P_{2k-4} \pmod {p}$. Thus, we obtain \[
P_{2k+2}\equiv -2P_{2k}+P_{2k-4}\pmod{p}.
\]
Now applying the induction hypothesis, we get
\[
\begin{aligned}
P_{2k+2}
&\equiv  (-1)^{k+1}\sum_{i=0}^{\lfloor k/3 \rfloor} (-1)^i \binom{k-2i}{i}2^{\,k+1-3i}\\
&+ (-1)^{k}\sum_{i=0}^{\lfloor (k-2)/3 \rfloor} (-1)^i \binom{k-2-2i}{i}2^{\,k-2-3i}\pmod{p}.
\end{aligned}
\]
By replacing \(i\) by \(i-1\) in the second sum and using Pascal's identity
$\binom{k+1-2i}{i}=\binom{k-2i}{i}+\binom{k-2i}{i-1},$ we obtain
\[
P_{2k+2} \equiv (-1)^{k+1}\sum_{i=0}^{\lfloor (k+1)/3 \rfloor} (-1)^i \binom{k+1-2i}{i}2^{\,k+1-3i}
\pmod{p},
\]
which completes the proof.
\end{proof}

\begin{proposition} \label{1} For the bi-periodic Padovan sequence with twin prime coefficients $a=p-2$ and $
b=p$, we have
\begin{equation*}
P_m \equiv 
\begin{cases}
(-1)^k \left(F_{k+3}-1\right) \pmod{p}, & \text{if $m=2k$},\\[1mm]
(-1)^{k-1} \left(F_{k+2}-1\right) \pmod{p}, & \text{if $m=2k+1$}.
\end{cases}
\end{equation*} 
\end{proposition}

\begin{proof} If $m=2k$, from Lemma \ref{l} and from the recurrence relation of the Fibonacci sequence, we have 
\[
\sum_{i=0}^{\lfloor k/3 \rfloor} (-1)^i \binom{k-2i}{i}2^{\,k-3i} = F_{k+3} - 1.
\] 
Thus, we get 
\[
P_{2k} \equiv (-1)^k \left(F_{k+3} - 1\right) \pmod{p}.
\]

Similarly, for $m = 2k+1$, using \eqref{p}, we have
$P_{2k+1} \equiv P_{2k-2} \pmod{p},$ which immediately yields the desired result.
\end{proof}

We recall below some known results concerning the Legendre symbol, the entry point and the Pisana period, which will be used throughout the paper. These definitions and results can be found in \cite{renault, savin, savinbook, vinson, wall}.

Let $a$ be an integer and let $p$ be an odd prime integer. An integer $a$ is a quadratic residue modulo $p$ if the congruence
$x^2\equiv a\pmod p$ has an integer solution; otherwise, $a$ is called a quadratic nonresidue modulo $p$. 
We denote by  $\left(\frac{a}{p}\right)$ the Legendre symbol of $a$ modulo $p$, defined by
$$\displaystyle\left(\dfrac{a}{p}\right)=\left\{
\begin{array}{rl}
1, & \hbox{if $p$ does not divide $a$ and $a$ is a quadratic residue mod $p$,} \\
-1, & \hbox{if $p$ does not divide $a$ and $a$ is a quadratic nonresidue mod $p$,} \\
0, & \hbox{if $p$ divides $a$.}
\end{array}
\right.
$$

Let $p$ be a positive odd prime. If $F_{z}$ is the smallest non-zero Fibonacci
number with the property $p\mid F_{z}$, then $z=z\left( p\right) $ is
defined as the \textit{entry point of }$p$ in the Fibonacci sequence. For a positive integer $m$, the \emph{Pisano period} $\pi(m)$ is defined as the smallest positive integer $k$ such that $F_{k} \equiv 0 \pmod{m}
\ \text{and}\ F_{k+1} \equiv 1 \pmod{m}$.

\noindent

\begin{lemma}\label{2}
The following statements hold:
\begin{itemize}
\item[(i)] $z(p)\mid \pi(p)$;
\item[(ii)] $\pi(p)=z(p)$ if $z(p)\equiv 2 \pmod{4}$;
\item[(iii)] $\pi(p)=2z(p)$ if $z(p)\equiv 0 \pmod{4}$;
\item[(iv)] $\pi(p)=4z(p)$ if $z(p)$ is odd;
\item[(v)] $p \mid F_m$ if and only if $z(p)\mid m$, for all $m\in\mathbb{N}$.
\end{itemize}
\end{lemma}

We are now ready to state one of our main results, which gives a complete characterization of when a bi-periodic Padovan quaternion with twin prime coefficients becomes a zero divisor in \( Q_{\mathbb{Z}_p}\).

\begin{theorem}\label{1}
Let $m$ be an even integer and $\frac{m}{2} \equiv -3 \pmod{z\left( p\right)}$. Then, the bi-periodic Padovan quaternion $QP_{m}$ with twin prime coefficients $a = p-2$ and $b = p$ is a zero divisor in the quaternion algebra $Q_{\mathbb{Z}_p}$ if and only if $p \equiv 1 \pmod{4}$ and $\frac{m}{2}\equiv s-2 \pmod{\pi\left( p\right)},$ where $s\in \left\{z\left(p\right)-1, 2z\left(p\right)-1, 3z\left(p\right)-1, 4z\left(p\right)-1\right\}$.
\end{theorem}

\begin{proof}
If $m$ is even, then $m=2k, k\in \mathbb{N}.$ A bi-periodic Padovan quaternion $QP_{m}$ is a zero divisor in the quaternion algebra $Q_{\mathbb{Z}_p}$ if and ony if the norm $N\left( QP_{m}\right) \equiv
0\left(\bmod \ p\right).$ Thus we have
\begin{eqnarray*}
N\left( QP_{m}\right)  &\equiv &0\left(\bmod \ p\right)  \\
&\Leftrightarrow &P_{m}^{2}+P_{m+1}^{2}+P_{m+2}^{2}+P_{m+3}^{2}\equiv
0\left(\bmod \ p\right)  \\
&\Leftrightarrow &P_{2k}^{2}+P_{2k+1}^{2}+P_{2k+2}^{2}+P_{2k+3}^{2}\equiv
0\left(\bmod \ p\right).   
\end{eqnarray*}
Using Proposition \ref{1}, we express the terms of the bi-periodic
Padovan sequence in terms of Fibonacci numbers. Therefore, we obtain
\begin{eqnarray*}
N\left( QP_{m}\right)  &\equiv &0\left(\bmod \ p\right)  \\
&\Leftrightarrow &\left. \left( \left( -1\right) ^{k}\left( F_{k+3}-1\right)
\right) ^{2}+\left( \left( -1\right) ^{k-1}\left( F_{k+2}-1\right) \right)
^{2}\right.  \\
&&\left. +\left( \left( -1\right) ^{k+1}\left( F_{k+4}-1\right) \right)
^{2}+\left( \left( -1\right) ^{k}\left( F_{k+3}-1\right) \right) ^{2}\equiv
0\left(\bmod \ p\right) \right.  \\
&\Leftrightarrow &2\left( F_{k+3}-1\right) ^{2}+\left( F_{k+2}-1\right)
^{2}+\left( F_{k+4}-1\right) ^{2}\equiv 0\left(\bmod \ p\right) .
\end{eqnarray*}
Since $\frac{m}{2}\equiv -3\left(\bmod \ z\left( p\right)\right) $ and from Lemma \ref{2}(v), it follows
\begin{equation}
z\left( p\right) \mid k+3\Leftrightarrow p\mid F_{k+3}\Leftrightarrow
F_{k+2}\equiv F_{k+4}\left(\bmod \ p\right)   \label{*}
\end{equation}
From here, we obtain that $\left( F_{k+2}-1\right) ^{2}\equiv \left(
F_{k+4}-1\right) ^{2}\left(\bmod \ p\right)$. Moreover, since $p\mid F_{k+3},$ we have $F_{k+3}^{2}\equiv 0\left(\bmod \ p\right) $. Therefore, we get
$$
N\left(QP_{m}\right)\equiv 0\left(\bmod \ p\right)
\Leftrightarrow
\left( F_{k+3}-1\right) ^{2}+\left( F_{k+2}-1\right) ^{2}\equiv 0\left( 
 \bmod \ p\right).$$ Since $p\mid F_{k+3}$, we have
$\left(F_{k+2}-1\right)^2\equiv-1\pmod p.$
It follows that $\left(\frac{-1}{p}\right)=1.$
Using the property of the Legendre symbol,
$\left(\frac{-1}{p}\right)=(-1)^{\frac{p-1}{2}},$
we conclude that 
$p\equiv 1\left(\bmod \ 4 \right).$
Therefore, we obtain
\begin{equation*}
N\left(QP_m\right)\equiv0\pmod p
\Leftrightarrow
\left(\frac{-1}{p}\right)=1
\Leftrightarrow
p\equiv1\pmod4.
\end{equation*}
\\
Moreover, a solution of the congruence $x^{2}\equiv -1 \left(\bmod \ p\right)$ is 
$F_{k+2}-1.$ Let $\alpha$ be an integer solution of this congruence such that $F_{k+2}\equiv \alpha +1 \left(\bmod \ p\right)$. Then there exist  $s\in\{0,1,\dots,\pi\left(p\right)-1\}$ such that 
$k\equiv s-2 \pmod{\pi\left( p\right)}$.
From the hypothesis, we have $\frac{m}{2} \equiv -3 \pmod{z\left( p\right)}\Leftrightarrow k \equiv -3 \pmod{z\left( p\right)}.$ 
By Lemma \ref{2}, we have
$z\left(p\right) \mid \pi\left(p\right),$ in fact $\pi\left(p\right)\in \left\{z\left(p\right),2z\left(p\right),4z\left(p\right)\right\}$. Hence, $$s-2\equiv -3 \pmod{z\left( p\right)} \Leftrightarrow
s\equiv z\left( p\right)-1 \pmod{z\left( p\right)}.$$
It follows that $s\in \left\{z\left(p\right)-1, 2z\left(p\right)-1, 3z\left(p\right)-1, 4z\left(p\right)-1\right\}$. \\\
\end{proof}

We now state the corresponding result for odd \(m\).

\begin{theorem}\label{2}
Let $m$ be an odd integer and $\frac{m-1}{2} \equiv -3 \pmod{z\left( p\right)}$. Then, the bi-periodic Padovan quaternion $QP_{m}$ with twin prime coefficients $a = p-2$ and $b = p$ is a zero divisor in the quaternion algebra $Q_{\mathbb{Z}_p}$ 
if and only if $p \equiv 1 \pmod{3}$ and 
$\frac{m-1}{2} \equiv s-2 \pmod{\pi(p)}$, where $s \in \left\{\,z(p)-1,\;2z(p)-1,\;3z(p)-1,\;4z(p)-1\,\right\}.$
\end{theorem}

\begin{proof}

If $m$ is odd, then $m=2k+1,\  k\in
\mathbb{N}$. Since $\frac{m-1}{2}\equiv -3\left(\bmod \ z\left( p\right)\right)$, we have $k\equiv -3\left(\bmod \ z\left( p\right)\right)$. From Lemma \ref{2} (v), it follows
\begin{equation}
z\left( p\right) \mid k+3\Leftrightarrow p\mid F_{k+3}\Leftrightarrow
F_{k+3}\equiv 0\left(\bmod \ p\right)   \label{**}
\end{equation}
A bi-periodic Padovan quaternion $QP_{m}$ is a zero divisor in the quaternion algebra $Q_{\mathbb{Z}_p}$ if and only if the norm 
$N\left( QP_{m}\right) \equiv
0\left(\bmod \ p\right)$. Thus we have
\begin{eqnarray*}
N\left( QP_{m}\right)  &\equiv &0\left(\bmod \ p \right)  \\
&\Leftrightarrow &P_{m}^{2}+P_{m+1}^{2}+P_{m+2}^{2}+P_{m+3}^{2}\equiv
0\left(\bmod \ p\right)  \\
&\Leftrightarrow &P_{2k+1}^{2}+P_{2k+2}^{2}+P_{2k+3}^{2}+P_{2k+4}^{2}\equiv
0\left(\bmod \ p\right)  \\
&\Leftrightarrow &P_{2k-2}^{2}+P_{2k+2}^{2}+P_{2k}^{2}+P_{2k+4}^{2}\equiv
0\left(\bmod \ p\right).
\end{eqnarray*}
Using Proposition \ref{1}, we obtain
\begin{eqnarray*}
N\left( QP_{m}\right)  &\equiv &0\left(\bmod \ p\right)  \\
&\Leftrightarrow &
\left( \left( -1\right) ^{k-1}\left( F_{k+2}-1\right) \right) ^{2}
+\left( \left( -1\right) ^{k}\left( F_{k+3}-1\right) \right) ^{2} \\
&&
+\left( \left( -1\right) ^{k+1}\left( F_{k+4}-1\right) \right) ^{2}
+\left( \left( -1\right) ^{k+2}\left( F_{k+5}-1\right) \right) ^{2}
\equiv 0\left(\bmod \ p\right)  \\
&\Leftrightarrow &
\left( F_{k+2}-1\right) ^{2}
+\left( F_{k+3}-1\right) ^{2}
+\left( F_{k+3}+F_{k+2}-1\right) ^{2} \\
&&
+\left( 2F_{k+3}+F_{k+2}-1\right) ^{2}
\equiv 0\left(\bmod \ p\right) .
\end{eqnarray*}
From (\ref{**}), we have
\begin{align*}
N\left( QP_{m} \right) &\equiv 0 \pmod{p} 
\;\Leftrightarrow\; 3 \left( F_{k+2}-1 \right)^2 \equiv -1 \pmod{p} \\[1mm]
&\Leftrightarrow \left( 3 \left( F_{k+2}-1 \right) \right)^2 \equiv -3 \pmod{p} \Leftrightarrow  \Big( \frac{-3}{p} \Big) = 1 \\[1mm]
&\Leftrightarrow \Big( \frac{-1}{p} \Big) \Big( \frac{3}{p} \Big) = 1 \Leftrightarrow\; \Big( \frac{3}{p} \Big) = (-1)^{\frac{p-1}{2}}.
\end{align*}
Since $p$ is a prime with $p\geq 5,$ by the quadratic
reciprocity law,
$\left(\frac{q}{p}\right)\left(\frac{p}{q}\right)
=
(-1)^{\frac{p-1}{2}\frac{q-1}{2}},$
with $q=3$, we obtain

$$\Bigl( \frac{3}{p} \Bigr) \Bigl( \frac{p}{3} \Bigr) =(-1)^{\frac{p-1}{2}}\Leftrightarrow \Bigl( \frac{3}{p} \Bigr) 
= (-1)^{\frac{p-1}{2}} \Bigl( \frac{p}{3} \Bigr).$$
Therefore, we obtain
\begin{equation*}
N\left( QP_{m}\right) \equiv 0\left(\bmod \ p\right) \Leftrightarrow \left( \frac{p}{3}\right) =1\Leftrightarrow p\equiv 1\left(\bmod \ 3\right).
\end{equation*}

Moreover, a solution of the congruence $x^{2}\equiv -3 \left(\bmod \ p\right)$
is $3\left(F_{k+2}-1\right).$ Let $\alpha$ be an integer solution of the congruence $x^{2}\equiv -3 \left(\bmod \ p\right)$ such that $3\left(F_{k+2}-1\right)\equiv \alpha \left(\bmod \ p\right).$
Since $p \neq 3$, the integer $3$ is invertible modulo $p$; let $\beta$ denote its inverse modulo $p$. It then follows that $F_{k+2}\equiv \beta \alpha +1\pmod{p}$. Consequently, there exist $s\in\{0,1,\dots,\pi\left(p\right)-1\}$ such that $k\equiv s-2 \pmod{\pi\left( p\right)}$.  
From the hypothesis, we have $\frac{m-1}{2} \equiv -3 \pmod{z\left( p\right)}\Leftrightarrow k \equiv -3 \pmod{z\left( p\right)}.$ 
By Lemma \ref{2}, we have
$z\left(p\right)\mid \pi\left(p\right)$, and in fact $\pi\left(p\right)\in \left\{z\left(p\right),2z\left(p\right),4z\left(p\right)\right\}$. Hence, $$s-2\equiv -3 \pmod{z\left( p\right)} \Leftrightarrow
s\equiv z\left( p\right)-1 \pmod{z\left( p\right)}.$$
It follows that $s\in \left\{z\left(p\right)-1, 2z\left(p\right)-1, 3z\left(p\right)-1, 4z\left(p\right)-1\right\}$. 
\end{proof}

%%%%%%%%%%%%%%%%%%%%%%%%%%%%%%%%%%%%%%%
\section{Bi-periodic Perrin quaternions over finite fields} \label{sec: 2}

In this section, we introduce the bi-periodic Perrin sequence and present its relationship with the bi-periodic Padovan sequence. Then, we define the bi-periodic Perrin quaternions and derive the corresponding quaternion identities. Finally, by fixing a specific twin prime parameterization, we study their norm structure and obtain criteria for the existence of zero divisors in \(Q_{\mathbb{Z}_p}\).\\

Let $\{P_n(a,b)\}_{n\geq0}$ denote the bi-periodic Padovan sequence $\{P_n\}_{n\geq0}$ defined in \cite{diskaya}. Using the same recurrence structure, we now introduce the bi-periodic Perrin sequence.

\begin{definition}
The bi-periodic Perrin sequence $\{R_n(a,b)\}_{n\ge 0}$, or simply $\{R_n\}_{n\ge 0}$, is defined by the recurrence relation
\[
R_n(a,b)=
\begin{cases}
a\,R_{n-2}(a,b)+R_{n-3}(a,b), & \text{if $n$ is even},\\[1mm]
b\,R_{n-2}(a,b)+R_{n-3}(a,b), & \text{if $n$ is odd},
\end{cases}
\qquad n\ge 3,
\] with initial values
$R_0(a,b)=3,\ R_1(a,b)=0,\ \text{and} \ R_2(a,b)=2.$
\end{definition}

The first few terms of these sequences are given in Table~\ref{tab:firstterms}.

\begin{table}[h!]
\small \centering
\[
\begin{array}{c l @{\qquad} l}
\hline
n & P_n & R_n \\
\hline
0 & 1 & 3 \\
1 & 0 & 0 \\
2 & a & 2 \\
3 & 1 & 3 \\
4 & a^{2} & 2a \\
5 & a+b & 2+3b \\
6 & 1+a^{3} & 3+2a^{2} \\
7 & a^{2}+ab+b^{2} & 2a+2b+3b^{2} \\
8 & a^{4}+2a+b & 2a^{3}+3a+3b+2 \\
9 & 1+a^{3}+a^{2}b+ab^{2}+b^{3}
  & 3+2a^{2}+2ab+2b^{2}+3b^{3} \\
10 & a^{5}+3a^{2}+2ab+b^{2}
   & 2a^{4}+3a^{2}+4a+3ab+2b+3b^{2} \\
\hline
\end{array}
\]
\caption{The first terms of the bi-periodic Padovan and the bi-periodic Perrin sequences.}
\label{tab:firstterms}
\end{table}
Unlike the classical case, the relation between the bi-periodic Padovan and Perrin sequences depends on the parity of the index. 

\begin{proposition} \label{padovan-perrin}
For $n\ge 3$, the bi-periodic Padovan sequence and bi-periodic Perrin sequence satisfy the following relation:
%\[R_n(a,b)=3\,P_{n-3}\!\left(a^{\xi(n+1)}b^{\xi(n)},b^{\xi(n+1)}a^{\xi(n)}\right)+2\,P_{n-2}\!\left(a^{\xi(n+1)}b^{\xi(n)},b^{\xi(n+1)}a^{\xi(n)}\right).\]
\[
R_n(a,b)=
\begin{cases}
3\,P_{n-3}(a,b)+2P_{n-2}(a,b), & \text{if $n$ is even},\\[1mm]
3\,P_{n-3}(b,a)+2P_{n-2}(b,a), & \text{if $n$ is odd},
\end{cases}\]
\end{proposition}
\begin{proof}
We prove the statement by induction on $n$. The cases $n=3$ and $n=4$ are verified directly.
Now assume that the statement holds for all indices up to $n=2m$, i.e.,
\[
R_{2m}(a,b)=3P_{2m-3}(a,b)+2P_{2m-2}(a,b)
\quad \text{and} \quad
R_{2m-1}(a,b)=3P_{2m-4}(b,a)+2P_{2m-3}(b,a).
\]

We prove the statement for $n=2m+1$ and $n=2m+2$.

Let $n=2m+1$. Using the recurrence relation of the bi-periodic Perrin sequence, we have
\[
R_{2m+1}(a,b)=bR_{2m-1}(a,b)+R_{2m-2}(a,b).
\]
Applying the induction hypothesis, we get
\[
R_{2m+1}(a,b)
= b\bigl(3P_{2m-4}(b,a)+2P_{2m-3}(b,a)\bigr)
+3P_{2m-5}(a,b)+2P_{2m-4}(a,b).
\]
Since $P_{2k+1}(a,b)=P_{2k+1}(b,a)$ for all $k\ge0$, and by using the definition of the bi-periodic Padovan sequence, we obtain
\[
\begin{aligned}
R_{2m+1}(a,b)
&=3\bigl(bP_{2m-4}(b,a)+P_{2m-3}(b,a)\bigr)
 +2\bigl(bP_{2m-3}(a,b)+P_{2m-4}(a,b)\bigr)\\
 &=3P_{2m-2}(b,a)+2P_{2m-1}(b,a).
\end{aligned}
\]
 
Let $n=2m+2$. Using the recurrence relation of the bi-periodic Perrin sequence, we have
\[
R_{2m+2}(a,b)=aR_{2m}(a,b)+R_{2m-1}(a,b).
\]
By the induction hypothesis, we get
\[
\begin{aligned}
R_{2m+2}(a,b)
&=a\bigl(3P_{2m-3}(a,b)+2P_{2m-2}(a,b)\bigr)
 +3P_{2m-4}(b,a)+2P_{2m-3}(b,a).
\end{aligned}
\]
Since $P_{2k+1}(a,b)=P_{2k+1}(b,a)$ for all $k\ge0$, and by using the definition of the bi-periodic Padovan sequence, we obtain
\[
\begin{aligned}
R_{2m+2}(a,b)
&=3\bigl(aP_{2m-3}(b,a)+P_{2m-4}(b,a)\bigr)
 +2\bigl(aP_{2m-2}(a,b)+P_{2m-3}(a,b)\bigr)\\
 &=3P_{2m-1}(a,b)+2P_{2m}(a,b).
\end{aligned}
\]

This completes the proof.
\end{proof}

To derive a unified relation for the quaternion case, covering both the odd and even cases, using the bi-periodic Padovan and Perrin sequences, we now define the bi-periodic Perrin quaternions as follows.
\begin{definition}
The bi-periodic Perrin quaternion sequence $\{QR_n(a,b)\}_{n\ge 0}$, or simply $\{QR_n\}_{n\ge 0}$, is defined as
\begin{equation*}
QR_{n}(a,b)=\begin{cases}
R_{n}(a,b)+R_{n+1}(b,a)i+R_{n+2}(a,b)j+R_{n+3}(b,a)k, & \text{if $n$ is even},\\[1mm]
R_{n}(b,a)+R_{n+1}(a,b)i+R_{n+2}(b,a)j+R_{n+3}(a,b)k, & \text{if $n$ is odd},
\end{cases}
\end{equation*}
where $R_{n}(a,b)$ is the $n$-th bi-periodic Perrin number.
\end{definition}
The first few terms of the bi-periodic Perrin quaternion sequence are given by
\[
\begin{aligned}
QR_0 &= 3+2j+3k,\\
QR_1 &= 2i+3j+2a\,k,\\
QR_2 &= 2+3i+2a\,j+(3a+2)k,\\
QR_3 &= 3+2a\,i+(3a+2)j+(2a^2+3)k,\\
QR_4 &= 2a+(3a+2)i+(2a^2+3)j+(3a^2+2a+2b)k,\\
QR_5&=(3a+2)+(2a^2+3)i+(3a^2+2a+2b)j+(2a^3+3a+3b+2)k.
\end{aligned}
\]

It is clear to see that bi-periodic Perrin quaternions also satisfy the same recurrence as bi-periodic Padovan quaternions given in (\ref{rec0}). By considering this recurrence relation, the generating function of the bi-periodic Perrin quaternion sequence is given by
\[
\sum_{n=0}^{\infty} QR_n(a,b)\,x^n
=
\frac{
A'(x)+B'(x)i+C'(x)j+D'(x)k
}{
1-(a+b)x^2+abx^4-x^6
}
\]
where
\[
\begin{aligned}
A'(x)&=3+(2-3a-3b)x^2+3x^3+b(3a-2)x^4+(2-3b)x^5,\\
B'(x)&=2x+3x^2-2bx^3+(2-3b)x^4+3x^5,\\
C'(x)&=2+3x-2bx^2+(2-3b)x^3+3x^4,\\
D'(x)&=3+2ax+(2-3b)x^2+(3-2ab)x^3+2x^5.
\end{aligned}
\]

Analogous to the relationship between the bi-periodic Padovan and bi-periodic Perrin sequences, the corresponding connection also holds for their quaternion extensions. This follows directly from the definition of bi-periodic Perrin quaternions and Proposition \ref{padovan-perrin}; hence, we omit the proof.

\begin{proposition} For all \(n\ge 3\),
$$QR_n(a,b)=3\,QP_{n-3}(a,b)+2\,QP_{n-2}(a,b).$$
\end{proposition}

In a similar manner, we consider bi-periodic Perrin quaternions with twin prime coefficients \( a = p-2 \) and \( b = p \) and state another main result providing a complete characterization of when a bi-periodic Perrin quaternion with twin prime coefficients becomes a zero divisor in \( Q_{\mathbb{Z}_p}\). Note that we exclude the case $p=181$ since in this case the discriminant of the quadratic congruence vanishes modulo $p$.
%%%%%%%%%%%%%%%%%%%%%%%%%%%%%%%%%%%%%%

\begin{theorem}\label{3}
Let $m$ be even and $\frac{m}{2} \equiv -3 \pmod{z\left( p\right)}$. Then, the bi-periodic Perrin quaternion $QR_{m}\left(a,b\right)$ with twin prime coefficients $a = p-2$ and $b = p$ for primes $p\neq 181$, is a zero divisor in the quaternion algebra $Q_{\mathbb{Z}_p}$ if and only if 
$$\frac{m}{2}\equiv s-1 \pmod{\pi(p)}, \
s\in\{z(p)-2,2z(p)-2,3z(p)-2,4z(p)-2\},$$
and one of the following holds:
\begin{itemize}
\item[(i)] $p\equiv 1,3 \pmod 8$ and $p$ is a quadratic residue modulo $181$;
\item[(ii)] $p\equiv 5,7 \pmod 8$ and $p$ is a quadratic non-residue modulo $181$.
\end{itemize}
\end{theorem}
\begin{proof}
If $m$ is even, then $m=2k, k\in \mathbb{N}.$ A bi-periodic Perrin quaternion $QR_{m}\left(p-2,p\right)$ is a zero divisor in the quaternion algebra $Q_{\mathbb{Z}_p}$ if and only if its norm satisfies $$N\left( QR_{m}\left(p-2,p\right)\right) \equiv 0\left(\bmod \ p\right).$$ 
By the definition of the norm of a quaternion and Definition~\ref{3}, we have
\begin{align*}
N\left( QR_{m}\right) \equiv &\ 0\left(\bmod \ p\right)  \\ 
\Leftrightarrow\ &R_{m}^{2}\left(p-2,p\right)+R_{m+1}^{2}\left(p,p-2\right)+R_{m+2}^{2}\left(p-2,p\right) +R_{m+3}^{2}\left(p,p-2\right)\equiv 0\left(\bmod \ p\right) \\
\Leftrightarrow\ & R_{2k}^{2}\left(p-2,p\right)+R_{2k+1}^{2}\left(p,p-2\right)+R_{2k+2}^{2}\left(p-2,p\right) +R_{2k+3}^{2}\left(p,p-2\right)\equiv 0\left(\bmod \ p\right).
\end{align*}
Using Proposition \ref{3}, we obtain
\begin{align*}
N\left( QR_{m}\right) \equiv & \ 0\left(\bmod \ p\right)  \\ 
\Leftrightarrow\ & \left(3P_{2k-3}+2P_{2k-2}\right)^{2}+\left(3P_{2k-2}+2P_{2k-1}\right)^{2} +\left(3P_{2k-1}+2P_{2k}\right)^{2}\\ &\quad+\left(3P_{2k}+2P_{2k+1}\right)^{2} \equiv 0\left(\bmod \ p\right) \\ \Leftrightarrow\ & 9\left(P^{2}_{2k-3}+P^{2}_{2k-2}+P^{2}_{2k-1}+P^{2}_{2k}\right) +4\left(P^{2}_{2k-2}+P^{2}_{2k-1}+P^{2}_{2k}+P^{2}_{2k+1}\right) \\ &\quad +12\left(P_{2k-3}P_{2k-2}+P_{2k-2}P_{2k-1}+P_{2k-1}P_{2k}+P_{2k}P_{2k+1}\right) \equiv 0\left(\bmod \ p\right) \\ \Leftrightarrow\ & 13\left(P^{2}_{2k-2}+P^{2}_{2k-1}+P^{2}_{2k}\right) +9 P^{2}_{2k-3}+4 P^{2}_{2k+1} \\ &\quad +12\left(P_{2k-3}P_{2k-2}+P_{2k-2}P_{2k-1}+P_{2k-1}P_{2k}+P_{2k}P_{2k+1}\right) \equiv 0\left(\bmod \ p\right). \end{align*}
Using (\ref{p}), we have
\begin{align*}
N\left( QR_{m}\right)
&\equiv 0\left(\bmod \ p\right)\\
&\Leftrightarrow  13\left(P^{2}_{2k-2}+P^{2}_{2k-4}+ P^{2}_{2k} \right)
+9 P^{2}_{2k-6}
+4 P^{2}_{2k-2} \\
&\quad + 12\left(P_{2k-6} P_{2k-2}+ P_{2k-2} P_{2k-4}
+ P_{2k-4}P_{2k}+ P_{2k} P_{2k-2}\right)\\
&\equiv 0\left(\bmod \ p\right).
\end{align*}
%We note that in this proof, the Padovan number $P_{j}=P_{j}\left(p-2,p\right),$ for all $j$ which appear in this proof (in accordance with Proposition \ref{3}).\\
Using Proposition (\ref{1}), we obtain
\begin{align*}
N\!\left(QR_m\right)
&\equiv 0 \pmod p \Leftrightarrow
17\left(\left(-1\right)^{k-1}\left(F_{k+2}-1\right)\right)^2 \\
&\quad +13\left(\left(-1\right)^{k-2}\left(F_{k+1}-1\right)\right)^2
+13\left(\left(-1\right)^k\left(F_{k+3}-1\right)\right)^2 \\
&\quad +9\left(\left(-1\right)^{k-3}\left(F_k-1\right)\right)^2 \\
&\quad +12\left(\left(-1\right)^k\left(F_k-1\right)\right)
   \left(\left(-1\right)^{k-1}\left(F_{k+2}-1\right)\right) \\
&\quad +12\left(\left(-1\right)^{k-1}\left(F_{k+2}-1\right)\right)
   \left(\left(-1\right)^{k-2}\left(F_{k+1}-1\right)\right) \\
&\quad +12\left(\left(-1\right)^{k-2}\left(F_{k+1}-1\right)\right)
   \left(\left(-1\right)^k\left(F_{k+3}-1\right)\right) \\
&\quad +12\left(\left(-1\right)^k\left(F_{k+3}-1\right)\right)
   \left(\left(-1\right)^{k-1}\left(F_{k+2}-1\right)\right)
\equiv 0 \pmod p.
\end{align*}
Since $\frac{m}{2} \equiv -3 \pmod{z(p)},$ we have
$z(p) \mid (k+3).$ By Lemma~\ref{2}(v), this is equivalent to $p \mid F_{k+3}.$ Therefore, we obtain
\begin{align*}
N\left( QR_{m}\right)
&\equiv 0\left(\bmod \ p\right)
\Leftrightarrow 17\left( F_{k+2}-1\right)^{2}
+13\left( F_{k+1}-1\right)^{2}
+13
+9\left( F_{k}-1\right)^{2} \\
&\quad -12\left( F_{k}-1\right)\left( F_{k+2}-1\right)
-12\left( F_{k+2}-1\right)\left( F_{k+1}-1\right)
-12\left( F_{k+1}-1\right)
+12\left( F_{k+2}-1\right) \\
&\equiv 0\left(\bmod \ p\right).
\end{align*}
Since $p \mid F_{k+3}$, we have 
$F_{k+3} \equiv 0 \pmod{p}.$ By the recurrence relation of the Fibonacci sequence, it follows that $F_{k+3} \equiv 0 \pmod{p}\Leftrightarrow F_{k+2}\equiv -F_{k+1}\left(\bmod \ p\right),$
and also $F_{k+3}\equiv 0\left(\bmod \ p\right) \Leftrightarrow F_{k}\equiv -2F_{k+1}\left(\bmod \ p\right)$.
Thus, we obtain
\begin{align*}
N\left( QR_{m}\right)
& \equiv 0\left(\bmod \ p\right)\\
& \Leftrightarrow 17\left( F_{k+1}+1\right)^{2}
+13\left( F_{k+1}-1\right)^{2}
+13
+9\left( 2F_{k+1}+1\right)^{2} \\
&\quad -12\left(-2F_{k+1}-1\right)\left( -F_{k+1}-1\right)
-12\left(- F_{k+1}-1\right)\left( F_{k+1}-1\right) \\
&\quad -12\left( F_{k+1}-1\right)
+12\left(- F_{k+1}-1\right) \equiv 0\left(\bmod \ p\right)\\
& \Leftrightarrow 54F_{k+1}^2 -16F_{k+1} + 28 \equiv 0 \pmod{p}.
\end{align*}
Since $p$ is a prime with $p \neq 2$, we have
\begin{align*}
N\left(QR_m\right) \equiv 0 \pmod{p}
&\Leftrightarrow
27F_{k+1}^2 - 8F_{k+1} + 14 \equiv 0 \pmod{p}.
\end{align*}
The congruence 
\begin{equation}\label{3.1}
27x^{2}-8x+14 \equiv 0\pmod{p}
\end{equation}
has an integer solution modulo $p$ if and only if its discriminant
$\Delta = -8 \cdot 181$ is a quadratic residue modulo $p$. Thus, $\left(\frac{\Delta}{p}\right) = \left(\frac{-8 \cdot 181}{p}\right) = 1,$ see \cite{savinbook}.\\
Using the multiplicativity of the Legendre symbol,
$\left(\frac{ab}{p}\right)
=
\left(\frac{a}{p}\right)
\left(\frac{b}{p}\right),$
we obtain
\[
\left(\frac{-8\cdot181}{p}\right)
=
\left(\frac{-1}{p}\right)
\left(\frac{2}{p}\right)
\left(\frac{181}{p}\right).
\]
Using the properties of the Legendre symbol
$\left(\frac{-1}{p}\right)=(-1)^{\frac{p-1}{2}}
\ \text{and}\
\left(\frac{2}{p}\right)=(-1)^{\frac{p^2-1}{8}},$
we obtain
\[
\left(\frac{181}{p}\right)
=
(-1)^{\frac{p^2-1}{8}}
(-1)^{\frac{p-1}{2}}
=
(-1)^{\frac{(p-1)(p+5)}{8}}.
\]
Finally, since $181\equiv1\pmod4$, the quadratic reciprocity law yields
\[
\left(\frac{181}{p}\right)=\left(\frac{p}{181}\right).
\]
Therefore, the congruence \eqref{3.1} has a solution if and only if
\[
\left(\frac{p}{181}\right)
=
(-1)^{\frac{(p-1)(p+5)}{8}}.
\]
This yields the following explicit conditions:
\[
\left(\frac{\Delta}{p}\right) = 1 \ \Leftrightarrow \ 
\begin{cases}
p \equiv 1,3 \pmod{8} \text{ and } p \text{ is a quadratic residue modulo } 181, \\
\text{or} \\
p \equiv 5,7 \pmod{8} \text{ and } p \text{ is a quadratic non-residue modulo } 181.
\end{cases}
\]

Moreover, a solution of the congruence (\ref{3.1}) is 
$F_{k+1},$ so we have
$$ 27F^{2}_{k+1} -8F_{k+1}+14 \equiv 0\pmod p
\Leftrightarrow \left(2\cdot 27F_{k+1}-8\right)^{2}\equiv \Delta \pmod p $$
$$\Leftrightarrow \left(2\cdot 27F_{k+1}-8\right)^{2}\equiv -8\cdot 181 \pmod p$$
$$\Leftrightarrow \left(27F_{k+1}-4\right)^{2}\equiv -2\cdot 181 \pmod p.$$
Let $\alpha$ be an integer solution of the congruence $x^{2}\equiv -2\cdot 181 \left(\bmod \ p\right)$ such that $27 F_{k+1}\equiv \alpha + 4 \left(\bmod \ p\right).$
Since $p\neq 3$, the integer $27$ is invertible modulo $p$; let $\beta$ denote its inverse. It then follows that $F_{k+1}\equiv \beta (\alpha+4) \pmod{p}$. Therefore, there exists an integer $s\in\{0,1,\dots,\pi\left(p\right)-1\}$ such that $k\equiv s-1 \pmod{\pi\left( p\right)}$.
On the other hand, from the hypothesis, we have $\frac{m}{2} \equiv -3 \pmod{z\left( p\right)}\Leftrightarrow k \equiv -3 \pmod{z\left( p\right)}.$ 
By Lemma \ref{2}, it results
$z\left(p\right)\mid \pi\left(p\right)$, and in fact $\pi\left(p\right)\in \left\{z\left(p\right),2z\left(p\right),4z\left(p\right)\right\}$. Therefore, we have $$s-1\equiv -3 \pmod{z\left( p\right)} \Leftrightarrow
s\equiv z\left( p\right)-2 \pmod{z\left( p\right)}.$$
It follows that $s\in \left\{z\left(p\right)-2, 2z\left(p\right)-2, 3z\left(p\right)-2, 4z\left(p\right)-2\right\}$. 
 \end{proof}
 
\medskip
We now consider the case $p=181$ in Theorem \ref{3}.

\begin{corollary}
Let \(m\) be an even positive integer such that 
\(\frac{m}{2}\equiv -3 \pmod{z(181)}\).
Then the bi-periodic Perrin quaternion \(QR_{m}(a,b)\) with twin prime coefficients
\(a=179\) and \(b=181\) is a zero divisor in the quaternion algebra
\(Q_{\mathbb{Z}_{181}}\) if and only if
$$\frac{m}{2}\equiv 47\left(\bmod \ 90\right).$$
\end{corollary}

\begin{proof}
If $m$ is even, then $m=2k, k\in \mathbb{N}.$
Analogously as in the proof of Theorem \ref{3}, we obtain that
the bi-periodic Perrin quaternion $QR_{m}\left(179,181\right)$ is a zero divisor in the quaternion algebra $Q_{\mathbb{Z}_{181}}$ if and only if 
$$\left(27F_{k+1}-4\right)^{2}\equiv 0\left(\bmod \ 181\right)
\Leftrightarrow 
27F_{k+1}\equiv 4\left(\bmod \ 181\right).$$
Since $\gcd(27,181)=1$, it follows that the last congruence has a unique solution modulo $181$ and this is $4\cdot 27^{\varphi\left(181\right)}\left(\bmod \ 181\right)=4\cdot 27^{180}\left(\bmod \ 181\right)\equiv 4\cdot 114\left(\bmod \ 181\right),$ where $\varphi$ is the Euler's function. So, we have
$$27F_{k+1}\equiv 4\left(\bmod \ 181\right)\Leftrightarrow 
F_{k+1}\equiv 4\cdot 114\left(\bmod \ 181\right)\Leftrightarrow 
F_{k+1}\equiv 94\left(\bmod \ 181\right).$$
From \cite{renault}, the Pisano period $\pi\left(181\right)=90$ and
$$F_{k+1}\equiv 94\left(\bmod \ 181\right)\Leftrightarrow 
k+1\equiv 48\left(\bmod \ 90\right) \Leftrightarrow 
k\equiv 47\left(\bmod \ 90\right)\Leftrightarrow 
\frac{m}{2}\equiv 47\left(\bmod \ 90\right).$$
\end{proof}

We now state the corresponding result for odd \(m\).

\begin{theorem}\label{4}
Let $m$ be an odd integer and $\frac{m-1}{2} \equiv -3 \pmod{z\left( p\right)}$. Then, the bi-periodic Perrin quaternion $QR_{m}\left(a,b\right)$ with twin prime coefficients $a = p-2$ and $b = p$ for primes $p \neq 7, 13$, is a zero divisor in the quaternion algebra $Q_{\mathbb{Z}_p}$ if and only if the Jacobi symbol satisfies $\left(\frac{p}{13 \cdot 239}\right) = 1$ and $\frac{m-1}{2} \equiv s-1 \pmod{\pi(p)}$, where $s \in \{ z(p)-2, 2z(p)-2, 3z(p)-2, 4z(p)-2 \}.$

\end{theorem}

\begin{proof}
If $m$ is odd, then $m=2k+1, k\in \mathbb{N}.$ A bi-periodic Perrin quaternion $QR_{m}\left(p-2,p\right)$ is a zero divisor in the quaternion algebra $Q_{\mathbb{Z}_p}$ if and ony if the norm $$N\left( QR_{m}\left(p-2,p\right)\right) \equiv
0\left(\bmod \ p\right).$$ 
By the definition of the norm of a quaternion and Definition~\ref{3}, we have
\begin{align*}
N\left( QR_{m}\right) \equiv & \left(\bmod \ p\right)\\
\Leftrightarrow\ & R_{m}^{2}\left(p,p-2\right)+R_{m+1}^{2}\left(p-2,p\right)+R_{m+2}^{2}\left(p,p-2\right)+R_{m+3}^{2}\left(p-2,p\right)\equiv 0\left(\bmod \ p\right) \\
\Leftrightarrow\ & R_{2k+1}^{2}\left(p,p-2\right)+R_{2k+2}^{2}\left(p-2,p\right)+R_{2k+3}^{2}\left(p,p-2\right)+R_{2k+4}^{2}\left(p-2,p\right)\equiv 0\left(\bmod \ p\right). 
\end{align*}
Using Proposition \ref{3}, we obtain
\begin{align*}
N\left( QR_{m}\right) \equiv 0 & \left(\bmod \ p\right)\\
\Leftrightarrow\ & \left(3P_{2k-2}+2P_{2k-1}\right)^{2}+ \left(3P_{2k-1}+2P_{2k}\right)^{2}+ \left(3P_{2k}+2P_{2k+1}\right)^{2}\\
&\quad+\left(3P_{2k+1}+2P_{2k+2}\right)^{2} \equiv 0\left(\bmod \ p\right) \\
\Leftrightarrow\ & 9\left(P^{2}_{2k-2}+P^{2}_{2k-1}+ P^{2}_{2k}+P^{2}_{2k+1} \right)+ 
4\left( P^{2}_{2k-1}+P^{2}_{2k}+ P^{2}_{2k+1}+P^{2}_{2k+2}\right) \\
&\quad + 12\left( P_{2k-2} P_{2k-1}+ P_{2k-1} P_{2k}+ P_{2k} P_{2k+1}+P_{2k+1} P_{2k+2}\right) \equiv 0\left(\bmod \ p\right) \\
\Leftrightarrow\ &13\left(P^{2}_{2k-1}+P^{2}_{2k}+ P^{2}_{2k+1} \right)+9 P^{2}_{2k-2}+4 P^{2}_{2k+2} \\
&\quad + 12\left(P_{2k-2} P_{2k-1}+ P_{2k-1} P_{2k}+ P_{2k} P_{2k+1}+ P_{2k+1} P_{2k+2}\right) \equiv 0\left(\bmod \ p\right).
\end{align*}
Using (\ref{p}), we have
$$N\left( QR_{m}\right)  \equiv 0\left(\bmod \ p\right)\Leftrightarrow 13\left(P^{2}_{2k-4}+P^{2}_{2k}+ P^{2}_{2k-2} \right)+9 P^{2}_{2k-2}+4 P^{2}_{2k+2}+$$
$$+ 12\left(P_{2k-2} P_{2k-4}+ P_{2k-4} P_{2k}+ P_{2k} P_{2k-2}+ P_{2k-2} P_{2k+2}\right)\equiv 0\left(\bmod \ p\right).$$
Using Proposition (\ref{1}), we obtain
\begin{align*}
N\left( QR_{m}\right) &\equiv 0\left(\bmod \ p\right) \Leftrightarrow\ 
13 \left( F_{k+1}-1\right)^{2}
+22 \left( F_{k+2}-1\right)^{2}
+13 \left( F_{k+3}-1\right)^{2}
+4 \left( F_{k+4}-1\right)^{2}\\
&\qquad -12 \left( F_{k+2}-1\right)\left( F_{k+1}-1\right)
+12 \left( F_{k+1}-1\right)\left( F_{k+3}-1\right)\\
&\qquad -12 \left( F_{k+3}-1\right)\left( F_{k+2}-1\right)
+12 \left( F_{k+2}-1\right)\left( F_{k+4}-1\right)
\equiv 0\left(\bmod \ p\right).
\end{align*}

Since  $\frac{m-1}{2} \equiv -3 \pmod{z\left( p\right)}$, we have $ z\left( p\right)\mid \left(k+3\right)$. By Lemma \ref{2}(v), this is equivalent to $p\mid F_{k+3}.$ Therefore, we obtain
\begin{align*}
N\left( QR_{m}\right)  &\equiv 0\left(\bmod \ p\right)\Leftrightarrow\\
&13\left( F_{k+1}-1\right)^{2}+
22\left( F_{k+2}-1\right)^{2}+13+ 4 \left( F_{k+4}-1\right)^{2}-12 \left( F_{k+2}-1\right) \left( F_{k+1}-1\right)\\
&\quad -12\left( F_{k+1}-1\right)+12\left( F_{k+2}-1\right)+12 \left( F_{k+2}-1\right) \left( F_{k+4}-1\right)
\equiv 0\left(\bmod \ p\right).
\end{align*}

Since $p\mid F_{k+3}$, we have
$F_{k+3}\equiv 0\left(\bmod \ p\right)$. From the recurrence relation of the Fibonacci
sequence, it follows that $F_{k+3}\equiv 0\left(\bmod \ p\right)\Leftrightarrow F_{k+2}\equiv -F_{k+1}\left(\bmod \ p\right)$. Then $F_{k+4}\equiv -F_{k+1}\left(\bmod \ p\right)$. Thus, we obtain
\begin{align*}
N\left( QR_{m}\right)  &\equiv 0\left(\bmod \ p\right)\Leftrightarrow\\
& 13\left( F_{k+1}-1\right)^{2}+
22\left( F_{k+1}+1\right)^{2}+13+ 4  \left( F_{k+1}+1\right)^{2}-12 \left( -F_{k+1}-1\right) \left( F_{k+1}-1\right)\\
&\quad -12\left( F_{k+1}-1\right)+12\left( -F_{k+1}-1\right)+12 \left( -F_{k+1}-1\right) \left( -F_{k+1}-1\right)
\equiv 0\left(\bmod \ p\right)\\
&\Leftrightarrow 63F^{2}_{k+1} +26F_{k+1}+52 \equiv 0\left(\bmod \ p\right).
\end{align*}

The congruence
\begin{equation}\label{3.2}
63x^{2} +26x+52 \equiv 0\left(\bmod \ p\right)
\end{equation}
has an integer solution modulo $p$ if and only if its discriminant $\Delta = -4\cdot13\cdot 239$ is a quadratic residue modulo $p$, i.e., $\left(\frac{\Delta}{p}\right) = \left(\frac{-4\cdot13\cdot 239}{p}\right)=1,$ see \cite{savinbook}. Note that $p \neq 239$, since in that case $p=239$ and $p-2 = 237$ would not form a twin prime pair.
%For the same reason $p\neq 2.$ From the hypothesis  $p\neq 13.$\\
Using the multiplicativity of the Legendre symbol, we have
$$\left(\frac{-4\cdot 13\cdot 239}{p}\right) = \left(-1\right)^{\frac{p-1}{2}}\left(\frac{13}{p}\right) \left(\frac{239}{p}\right).$$
Using the quadratic reciprocity law, we have
$\left(\frac{13}{p}\right)= \left(\frac{p}{13}\right)$ and $\left(\frac{239}{p}\right)=\left(-1\right)^{\frac{p-1}{2}} \left(\frac{p}{239}\right)$. By combining these results, we obtain
$\left(\frac{\Delta}{p}\right)=\left(\frac{p}{13}\right) \left(\frac{p}{239}\right).$
So, $$\left(\frac{\Delta}{p}\right)=1\Leftrightarrow \left(\frac{p}{13}\right)\left(\frac{p}{239}\right)=1.$$
Using the fact that the Jacobi symbol $\left(\frac{p}{13\cdot 239}\right)$ is a product of two Legendre symbols, namely $\left(\frac{p}{13\cdot 239}\right)=\left(\frac{p}{13}\right) \left(\frac{p}{239}\right),$ we obtain
\begin{equation*}
\left(\frac{\Delta}{p}\right)=1\Leftrightarrow\left(\frac{p}{13\cdot 239}\right)=1.
\end{equation*}

Moreover, a solution of the congruence (\ref{3.2}) is $F_{k+1},$ so we have
$$63F^{2}_{k+1} +26F_{k+1}+52 \equiv 0\left(\bmod \ p\right)
\Leftrightarrow \left(2\cdot63\cdot F_{k+1}+26\right)^{2}\equiv -4\cdot 13\cdot 239 \left(\bmod \ p\right)$$
$$\Leftrightarrow \left(63 F_{k+1}+13\right)^{2}\equiv -13\cdot 239 \left(\bmod \ p\right)
\Leftrightarrow \left(63 F_{k+1}+13\right)^{2}\equiv -3107 \left(\bmod \ p\right).$$
Let $\alpha$ be an integer solution of the congruence $x^{2}\equiv -3107 \left(\bmod \ p\right)$ such that $63 F_{k+1}\equiv \alpha -13 \left(\bmod \ p\right).$ Since $p\neq 3$ and $p\neq 7$, the integer $63$ is invertible modulo $p$. Let $\beta$ denote its inverse modulo $p$. Then it follows that $F_{k+1}\equiv \beta (\alpha-13) \pmod{p}.$
Consequently, there exists an integer $s\in\{0,1,\dots,\pi\left(p\right)-1\}$ such that
$k\equiv s-1 \pmod{\pi\left( p\right)}.$
From the hypothesis, we have $\frac{m-1}{2} \equiv -3 \pmod{z\left( p\right)}\Leftrightarrow k \equiv -3 \pmod{z\left( p\right)}.$ 
By Lemma \ref{2}, it results
$z\left(p\right)\mid \pi\left(p\right)$, and in fact $\pi\left(p\right)\in \left\{z\left(p\right),2z\left(p\right),4z\left(p\right)\right\}$. Hence, we have $$s-1\equiv -3 \pmod{z\left( p\right)} \Leftrightarrow
s\equiv z\left( p\right)-2 \pmod{z\left( p\right)}.$$
It follows that $s\in \left\{z\left(p\right)-2, 2z\left(p\right)-2, 3z\left(p\right)-2, 4z\left(p\right)-2\right\}$. 
\end{proof}

We now consider the case $p=7$ in Theorem \ref{4}. 

\begin{corollary}\label{2}
Let \(m\) be an odd integer such that 
\(\frac{m-1}{2}\equiv -3 \pmod{z(7)}\).
Then the bi-periodic Perrin quaternion \(QR_{m}(5,7)\) is a zero divisor in the quaternion algebra
\(Q_{\mathbb{Z}_{7}}\) if and only if
$$\frac{m-1}{2}\equiv 4, 10\left(\bmod \ 16\right).$$  
\end{corollary}

\begin{proof}
If $m$ is odd, then $m=2k+1, k\in \mathbb{N}.$
Analogously as in the proof of Theorem \ref{4}, we obtain that
a bi-periodic Perrin quaternion $QR_{m}\left(5,7\right)$ with  is a zero divisor in the quaternion algebra $Q_{\mathbb{Z}_{7}}$ if and only if 
\begin{align*}
63F^{2}_{k+1} +26F_{k+1}+52 &\equiv 0\left(\bmod \ 7\right) \Leftrightarrow 
5F_{k+1}\equiv 4\left(\bmod \ 7\right).
\end{align*}
Since $\gcd(5,7)=1$, the last congruence is equivalent to
$$F_{k+1}\equiv 3\cdot 4\left(\bmod \ 7\right)\Leftrightarrow F_{k+1}\equiv 5\left(\bmod \ 7\right).$$
From \cite{renault}, we have
$$ k+1\equiv 5, 11\left(\bmod \ 16\right)
\Leftrightarrow k \equiv 4, 10\left(\bmod \ 16\right)
\Leftrightarrow \frac{m-1}{2} \equiv 4, 10\left(\bmod \ 16\right).$$
\end{proof}

We now examine the application of Theorem \ref{4} to the case \(p=13\). Somewhat surprisingly, we show that in this setting all bi-periodic Perrin quaternions are invertible in the quaternion algebra \(Q_{\mathbb{Z}_{13}}\).

\begin{corollary}\label{3}
Let \(m\) be an odd integer such that 
\(\frac{m-1}{2}\equiv -3 \pmod{z(13)}\).
Then all bi-periodic Perrin quaternions \(QR_{m}(11,13)\) are invertible in the quaternion algebra
$Q_{\mathbb{Z}_{13}}$.
\end{corollary}

\begin{proof}
If $m$ is odd, then $m=2k+1,$ $k\in \mathbb{N}.$
It is known that a quaternion is invertible in the quaternion algebra $Q_{\mathbb{Z}_{13}}$
if and only if it is not a zero divisor in the quaternion algebra $Q_{\mathbb{Z}_{13}}.$
Suppose, by contradiction, that there exists a bi-periodic Perrin quaternion \(QR_{m}(11,13)\) which is a zero divisor in the quaternion algebra $Q_{\mathbb{Z}_{13}}.$ Then, as in the proof of Theorem \ref{4}, \(F_{k+1}\) must satisfy the congruence (\ref{3.2}), i.e.,
$63 F_{k+1}^2 + 26 F_{k+1} + 52 \equiv 0 \pmod{p}.$
Reducing modulo $13$, this simplifies to
$11 F_{k+1}^2 \equiv 0 \pmod{13}$. Multiplying both sides by $6$, the inverse of $11$ modulo $13$, we obtain $F_{k+1}^2 \equiv 0 \pmod{13}$ which implies $F_{k+1} \equiv 0 \pmod{13}.$ Thus, this is the only possible solution.

From the hypothesis, \(z(13) \mid (k+3)\), so \(13 \mid F_{k+3}\). 
By the Fibonacci recurrence, this implies \(13 \mid F_{k+2}\) and then by induction, all Fibonacci numbers would be divisible by 13, which is impossible. Hence, our assumption is false and every \(QR_m(11,13)\) is invertible in \(Q_{\mathbb{Z}_{13}}\).
\end{proof}

%%%%%%%%%%%%%%%%%%%%%%%%%%%%%%%%
\section{Conclusion}

In this paper, we introduced the bi-periodic Perrin sequence and established its relationship with the bi-periodic Padovan sequence. Extending these sequences to the quaternion setting, we analyzed their norms in the split algebra \(Q_{\mathbb{Z}_p}\). The main contribution of this work is the explicit characterization of zero divisors for bi-periodic Padovan and Perrin quaternions in $Q_{\mathbb{Z}_p}$. A key strength of our approach is that it applies to a general prime \(p\), rather than being limited to specific values. Within this setting, we explicitly characterized which bi-periodic Padovan and Perrin quaternions are zero divisors in \(Q_{\mathbb{Z}_p}\).  In particular, for \(p=13\), we showed that all bi-periodic Perrin quaternions of odd index $m$, satisfying $\frac{m-1}{2}\equiv -3 \pmod{z(13)},$ are invertible in quaternion algebra \(Q_{\mathbb{Z}_p}\).

As directions for future research, it would be of interest to investigate quaternion algebras associated with higher-dimensional $k$-periodic recurrence sequences for $k>2$, where the bi-periodic case corresponds to $k=2$. It would also be interesting to consider other choices of the coefficients $a$ and $b$ beyond the twin prime setting and to investigate whether analogous zero-divisor criteria can be obtained for other periodic recurrence sequences and related quaternion algebras.

\section*{Acknowledgements}
The authors thank the referee for several helpful comments. Elif Tan gratefully acknowledges the Transilvania University of Bra\c{s}ov for supporting her visit in January 2026.

%%%%%%%%%%%%%%%%%%%%%%%%%%%%%%%%

\end{document}